\documentclass[12pt]{article}

\usepackage{amstext}
\usepackage{amsthm}
\usepackage{amsmath}
\usepackage{amssymb}
\usepackage{latexsym}
\usepackage{amsfonts}
\usepackage{color}
\usepackage{graphicx}

\usepackage[backrefs]{amsrefs}

\usepackage[pagebackref,hypertexnames=true, colorlinks, citecolor=black, linkcolor=blue, urlcolor=red]{hyperref}

\usepackage{latexsym}
\usepackage{amssymb}
\usepackage{euscript}

\let\cal=\mathcal      

\def\mcc{M\raise.5ex\hbox{c}C}
\def\mccarthy{M\raise.5ex\hbox{c}Carthy}

\def\eg{{\it e.g. }}
\def\ie{{\it i.e. }}

\def\h{{\cal H}}

\def\K{{\cal K}}
\def\M{{\cal M}}
\def\N{{\cal N}}

\let\i=\infty

\def\={\ = \ }    
\def\ot{\otimes}

\def\C{\mathbb C}
\def\R{\mathbb R}

\def\dis{\displaystyle}

\def\be{\setcounter{equation}{\value{theorem}} \begin{equation}}
\def\ee{\end{equation} \addtocounter{theorem}{1}}
\def\beq{\begin{eqnarray*}}
\def\eeq{\end{eqnarray*}}
 
\def\att{\addtocounter{theorem}{1}}
\def\vs{\vskip 5pt}
\def\bs{\vskip 12pt}

\def\exam{\bs \att {{\bf Example \thetheorem \ }} }

\def\ep{{}{\hfill $\Box$} \vskip 5pt \par}

\def\bl{\begin{lemma}}
\def\el{\end{lemma}}
\def\bt{\begin{theorem}}
\def\et{\end{theorem}}
\def\bprop{\begin{prop}}
\def\eprop{\end{prop}}
\def\bd{\begin{definition}}
\def\ed{\end{definition}}
\def\br{\begin{remark}}
\def\er{\end{remark}}
\def\bexer{\begin{exercise}}
\def\eexer{\end{exercise}}

\newtheorem{theorem}{Theorem}[section]
\newtheorem{prop}[theorem]{Proposition}
\newtheorem{lemma}[theorem]{Lemma}

\newtheorem{definition}[theorem]{Definition}

\renewcommand\O{\Omega}
\def\gdel{G_\delta}

\def\set#1#2{\{ #1 \, | \, #2\}}
\def\norm#1{\| #1 \|}

\def\L{{\mathcal L}}
\def\id{{\rm id}}

\def\M{{\mathbb M}}
\def\bh{{B(\h)}}
\def\bhd{\bh^d}

\def\U{\mathcal U}

\renewcommand\N{{\mathbb N}}
\def\mn{\M_n}
\def\mnd{\mn^d}
\def\mmd{{\mathbb M}_m^d}

\def\md{{\mathbb M}^{[d]}}

\def\gds{G_\delta^\sharp}

\def\d{\delta}
\def\norm#1{\| #1 \|}

\def\pd{{\mathbb P}^d}

\def\idcn{{\rm id}_{\mathbb{C}^n}}

\def\gds{G_\delta^\sharp}
\def\lhd{{\mathcal L}(\h)^d}
\def\lh{{\mathcal L}(\h)}
\def\lho{{\mathcal L}_1(\h)}

\numberwithin{equation}{section}
\title{Aspects of Non-commutative Function Theory
}
\author{
Jim Agler
\thanks{Partially supported by National Science Foundation Grant
DMS 1361720}\\
U.C. San Diego\\ 
La Jolla, CA 92093
\and
John E. M\raise.5ex\hbox{c}Carthy
\thanks{Partially supported by National Science Foundation Grant  
DMS 1300280
}
\\ 
Washington University\\
 St. Louis, MO 63130
}

\begin{document}

\bibliographystyle{plain}
\maketitle

\begin{abstract}
We discuss non commutative functions, which naturally arise when
dealing with functions of more than one matrix variable.
\end{abstract}

\section{Motivation}

A non-commutative polynomial  is an element of the algebra over
a free monoid;  an example is 
\be
\label{eqa1}
p(x,y) \= 2x^2 + 3 xy - 4 yx + 5 x^2y + 6 xyx.
\ee
Non-commutative function theory is the study of functions of non-commuting variables, which may be more
general than non-commutative polynomials. It is based on the observation that matrices are natural objects on which to evaluate an expression like \eqref{eqa1}. 

\subsection{LMI's}

A linear matrix inequality (LMI) is an inequality of the form
\be
\label{eqa2}
A_0 + \sum_{i=1}^M x_i A_i \geq  0.
\ee
The $A_i$ are given self-adjoint $n$-by-$n$ matrices, and the object is to find $x \in \R^m$ such that
\eqref{eqa2} is satisfied (or to show that it is never satisfied). LMI's are very important in control theory, and
there are efficient algorithms for finding solutions. See for example the book \cite{befb94}.

Often, the stability of a system is equivalent to whether a certain matrix valued function $F(x)$ is positive semi-definite;
but the function $F$ may be non-linear. A big question is when the inequality
\[
F(x) \ \geq \ 0 \]
can be reduced to an LMI. This has been studied, for example, in \cite{hmpv09, hm12, hkm12, hkm12b}.
Let us consider a simple example, the Riccati matrix inequality:
\be
\label{eqa3}
AX + XA^* - X BB^*X + C^*C \ >  \ 0 .
\ee
All the matrices in \eqref{eqa3} are real, and $A, B, C$ are known. The self-adjoint matrix $X$ is unknown. 
The quadratic inequality \eqref{eqa3} can be transformed into the following LMI:
\be
\label{eqa4}
\begin{pmatrix}
AX + XA^* + C^* C & XB \\
B^* X & I 
\end{pmatrix}
\ > \ 0 .
\ee
Suppose each matrix $A,B,C,X$ is $3$-by-$3$. Then, in terms of the matrix entries of $X$, \eqref{eqa4} is an LMI in 
6 variables. If one uses Sylvester's criterion to determine positive definiteness, one gets that \eqref{eqa4} is equivalent to
the positivity of 6 polynomials, of degrees 1 through 6, in these 6 variables.
This could then become a question in real algebraic geometry.

However, this approach has two obvious flaws: it loses the matrix structure, and the complexity increases rapidly with the
dimension. Can one study inequalities like \eqref{eqa4} in a dimension independent way?

\subsection{Non-commutative sums of squares}

Hilbert's seventeenth problem asked whether a non-negative polynomial in real variables could be written 
as a sum of squares of rational functions (Hilbert knew that a sum of squares of polynomials did not in general suffice).
Although the answer was proved to be yes by E. Artin in 1927, such problems are still an active area of research in real algebraic geometry today, and operator theory and functional analysis have played a r\^ole in this field \cite{sch91,put93}.

Let $\M_n$ denote the $n$-by-$n$ matrices.
J.W. Helton asked whether a non-commutative polynomial $p$  in the $2d$ variables
$x_1, \dots, x_d, x_1^*, \dots, x_d^*$ that is formally self-adjoint, and has the property that
\[
p(x_1, \dots,  x_d, x_1^*, \dots, x_d^*) \geq 0 \quad \forall\ n,\ \forall \ x_1, \dots , x_d \in \M_n 
\]
must have a representation
\[
p(x) \= \sum_{j} q_j(x_1, \dots , x_d)  q_j(x_1, \dots , x_d)^*
\]
as a (finite) sum of  squares, where each $q_j$ is a non-commutative polynomial in $d$ variables.
He and S. McCullough proved \cite{helt02,helmccu04} that the answer is yes.

This illustrates an important theme: if something is true at all matrix levels (a strong assumption),
it is true for obvious algebraic reasons (a strong conclusion), and the proof is often easier than in the scalar case.

\subsection{Implicit Function Theorem}

Consider the matrix equation \be
\label{eqa5}
X^3 + 2 X^2 Y + 3 Y X + 4 X  + 5 Y + 6 \= 0.
\ee
If $X$ is similar to a diagonal matrix, it is obvious that one can find solutions $Y$ to \eqref{eqa5}
that commute with it. What is perhaps surprising is that generically, this is all that happens:

\bprop
\label{pra2}
For a generic choice of $X \in \mn$, the only $Y$'s that satisfy \eqref{eqa5}
commute with $X$.
\eprop

To see why this is true (and what we mean by generic) let $p (X,Y) = X^3 + 2 X^2 Y + 3 Y X + 4 X  + 5 Y + 6$,
and consider the partial derivative of $p$ w.r.t. $Y$ in the direction $H$. This is
\beq
\frac{\partial}{\partial Y} p(X,Y) [ H] &\=&
\lim_{ t \to 0} \frac{1}{t} \left[ p(X, Y+tH) - p(X,Y) \right] \\
&=& 2 X^2 H + 3 HX + 5 H .
\eeq
An implicit function theorem would say that we can write $Y$ locally as a function of $X$
(and therefore as a matrix that commutes with $X$)
provided $\dis \frac{\partial}{\partial Y} p(X,Y)$ is full rank.
Ontologically, the partial derivative is a linear map from $\mn$ to $\mn$. 
It is therefore full rank if and only if it has 
no kernel. We can analyze this case using the following theorem of Sylvester from 1884 \cite{syl}:

\bt 
Let $A, B$ be in $\mn$.
There is a non-zero matrix $H$ satisfying the equation $AH - H B = 0$ if and only if
$\sigma(A) \cap \sigma(B)$ is non-empty.
\et
So $\dis \frac{\partial}{\partial Y} p(X,Y)$ is full rank if and only if 
\be
\label{eqa6}
\sigma(2 X^2 ) \cap \sigma(-3X - 5I) \ = \ \emptyset .
\ee
Condition \eqref{eqa6} in turn is generically true, so for such $X$ Proposition~\ref{pra2}
will follow from a non-commutative implicit function, such as Theorem \ref{thmc1} below.

\subsection{Functional Calculus}

If $T$ is a single operator on a Banach space $E$, the functional calculus is an extension of the evaluation map
\beq
\pi : \C [z] & \ \to & \L(E) \\
p  &\mapsto & p(T) 
\eeq
to a larger algebra than the polynomials. The Riesz-Dunford functional calculus gives the unique extension
to the algebra of all functions that are holomorphic on a neighborhood of $\sigma (T)$.
If now $T = (T^1, \dots, T^d)$ is a commuting $d$-tuple of operators, then the Taylor functional calculus
extends the evaluation map from polynomials (in $d$ commuting variables) to  functions holomorphic on a neighborhood
of the Taylor spectrum of $T$ \cite{tay70a, tay70b}. 

What about the case that $T$ is a $d$-tuple of non-commuting operators on $X$? Two distinct questions arise.

(i) What should play the r\^ole of holomorphic functions in the non-commutative setting?

(ii) How can they be applied to elements of $\L(E)^d$?

The answers to both questions have their roots in work of Taylor shortly after he completed his work on commuting 
operators \cite{tay72,tay73}. The answer to the first question is nc-functions, which we shall discuss in Section~\ref{secb},
and is the subject of the recent book \cite{kvv14} by D. Kaliuzhnyi-Verbovetskyi and V. Vinnikov.
The second question is discussed in Section~\ref{secfree}.

\subsection{Other motivations}

There has been an upwelling of interest in non-commutative function theory recently. In addition
to the motivations above, let us mention the work of
 Voiculescu \cite{voi4}, in the context of free probability;
Popescu \cite{po06,po08,po10,po11}, in the context of extending classical function theory to 
$d$-tuples of bounded operators;
 Ball, Groenewald and Malakorn \cite{bgm06}, in the context of extending realization formulas
from functions of commuting operators to functions of non-commuting operators;
Alpay and Kalyuzhnyi-Verbovetzkii \cite{akv06} in the context of
realization formulas for rational functions that are $J$-unitary on the boundary of the domain; and
Helton, Klep and McCullough \cite{hkm11a,hkm11b}
and Helton and McCullough \cite{hm12} in the context of developing a descriptive theory of the domains on
which LMI and semi-definite programming apply;  Muhly and Solel \cite{ms13}, in the context of tensorial function theory;
Cimpric,  Helton,  McCullough and Nelson \cite{chmn13} in the context of non-commutative real
Nullstellens\"atze;  the second author and  Timoney \cite{mt15} and
of Helton, Klep, McCullough and Slinglend,
in \cite{hkms09} on non commutative automorphisms; and 
the work of Pascoe and Tully-Doyle \cite{ptd13} on non-commutative operator monotonicity.

The book  \cite{kvv14} by D. Kaliuzhnyi-Verbovetskyi and V. Vinnikov is a much more
comprehensive treatment of the subject than we can give here, and also contains extensive historical references.

\section{NC-functions}
\label{secb}

We want to evaluate functions on $d$-tuples of matrices, where all the matrices in a given $d$-tuple are the same
dimension, but we want to allow this dimension to vary. So our universe becomes
\[
\md \= \cup_{n=1}^\i \mnd .
\]
We shall call a function $f$  defined on a subset of $\md$ {\em graded} if, whenever 
$x \in \mnd$, then $f(x) \in \mn$. If $x \in \mnd$ and $y \in \mmd$, we shall let $x \oplus y$ denote
the $d$-tuple in $\M_{n+m}$

\bd An nc-function $f$ on a set $\Omega \subseteq \md$ is a graded function that respects direct sums and joint similiarities, \ie 
\beq
f(x \oplus y) &\= &  f(x) \oplus f(y)\\
f(s^{-1} x s) := f (s^{-1} x^1 s, \dots, s^{-1} x^d s) &= & s^{-1}f(x) s, 
\eeq
where the equations are only required to hold when the arguments on both sides are in $\Omega$.
\ed

Notice that every non-commutative polynomial is an nc-function on $\md$. To get utility from the definition, one
needs $\Omega$ to have some structure.

\bd
A   set $\Omega \subseteq \md$ is an nc-set if
\newline
(i) If $x,y \in \Omega$ then $x \oplus y \in \Omega$.
\newline
(ii) For every $n \in \N$, $\Omega \cap \mnd$ is open in $\mnd$.
\newline
(iii) If $x \in \Omega \cap \mnd$ and $u$ is an $n$-by-$n$ unitary, then
$u^* x u \in \Omega$.
\ed

The algebraic properties of being nc on an nc-set have analytic consequences. For example,
J.W. Helton, I. Klep and S. McCullough in  \cite{hkm11b} proved that continuity implies analyticity (continuity can actually 
be weakened to local boundedness). 
\bt
\label{thmhkm}
Let $\O$ be an nc-set, let $f$ be an nc-function on $\O$, and assume $f$ is locally bounded on $\O \cap \mnd$ for every $n$. 
Then $f$ is an analytic function of the entries of the matrices at each level $n$.
\et
\begin{lemma}\label{lem2.20}(Lemma 2.6 in \cite{hkm11b}). 
Let $\O$ be an nc-domain in $\md$,  and let $f$ be an  nc-function on $\O$. 
Fix $n \ge 1$ and $v \in \M_n$. If $x,y \in \O \cap \mn^d$ and
\[
\begin{bmatrix}y&yv-vx\\ 0&x\end{bmatrix} \in \O \cap \M_{2n}^d,
\]
then
\be\label{2.160}
f(\begin{bmatrix}y&yv-vx\\ 0&x\end{bmatrix}) = \begin{bmatrix}f(y)&f(y)v-vf(x)\\ 0&f(x)\end{bmatrix}.
\ee
\end{lemma}
\begin{proof}
Let
$$s = \begin{bmatrix}\idcn&v\\0&\idcn\end{bmatrix}$$
so that
$$\begin{bmatrix}y&yv-vx\\ 0&x\end{bmatrix} = s^{-1}\begin{bmatrix}y&0\\ 0&x\end{bmatrix}s.$$
Since $f$ is nc,
\beq
f(\begin{bmatrix}y&yv-vx\\ 0&x\end{bmatrix}) 
&\ =\ & f(s^{-1}(y \oplus x)s)  \\ 
&=& s^{-1}(f(y) \oplus f(x))s \\ 
&=& \begin{bmatrix}\idcn &-v \\0&\idcn \end{bmatrix}
\begin{bmatrix}f(y)&0\\ 0&f(x)\end{bmatrix} \begin{bmatrix}\idcn &v \\0&\idcn \end{bmatrix}\\ 
&=&\begin{bmatrix}f(y)&f(y)v-vf(x)\\ 0&f(x)\end{bmatrix}.
\eeq
\end{proof}
{\sc Proof of Theorem~\ref{thmhkm}.}
First, we show that local boundedness implies continuity.
Fix $x \in \O \cap \mn^d$ and let $\varepsilon > 0$. Choose $r>0$ so that the ball 
\[
B({\begin{bmatrix}x&0\\0&x\end{bmatrix}},{r} )\subseteq \O \cap \M_{2n}^d.
\]
If  $r_1$ is chosen with $0<r_1<r$ then as $B({x\oplus x},{r_1})^-$ is a compact subset
 of $\O \cap \M_{2n}^d$ and $f$ is assumed locally bounded, there exists a constant $C$ such that
\be
\label{3.30}
z\in B({\begin{bmatrix}x&0\\0&x\end{bmatrix}},{r_1})  \implies \|f(z) \| < C.
\ee
Choose $\delta$ sufficiently small so that $\delta< \min\{r_1\varepsilon/2C,r_1/2\}$ 
and $B({x},{\delta} )\subseteq \O$. That $f$ is continuous at $x$ follows from the following claim.
\be\label{3.40}
y \in B({x},{\delta}) \implies f(y) \in B({f(x)},{\varepsilon}).
\ee
To prove the claim fix $y \in \mn^d$ with $\norm{y-x} < \delta$. Then $\norm{y-x} <r_1/2$ and $\norm{(C/\varepsilon)(y-x)} <r_1/2$. Hence by the triangle inequality,
\[
\norm{\begin{bmatrix}y&c(y-x)\\ 0&x\end{bmatrix}-\begin{bmatrix}x&0\\ 0&x\end{bmatrix} }< r_1.
\]
Hence, by \eqref{3.30},
\[
\norm{f(\begin{bmatrix}y&(C/\varepsilon)(y-x)\\ 0&x\end{bmatrix})}< C.
\]
But $x$, $y$, and $\begin{bmatrix}y&c(y-x)\\ 0&x\end{bmatrix}$ are in $\O$, so by Lemma \ref{lem2.20},
\[
f(\begin{bmatrix}y&(C/\varepsilon)(y-x)\\ 0&x\end{bmatrix})=\begin{bmatrix}f(y)&(C/\varepsilon)(f(y)-f(x))\\ 0&f(x)\end{bmatrix}
\]
In particular, we see that $\norm{(C/\varepsilon)(f(y)-f(x))} < C$, or equivalently, $f(y) \in B({f(x)},{\varepsilon})$. This proves \eqref{3.40}.

To see that $f$ is holomorphic, fix $x\in \O \cap \mn^d$. If $h \in \mn^d$ is selected sufficiently small, then
\[
\begin{bmatrix}x + \lambda h & h \\0&x\end{bmatrix} \in \O \cap \M_{2n}^d 
\]
for all sufficiently small $\lambda \in \C$. But
\[
\begin{bmatrix}x + \lambda h & h \\0&x\end{bmatrix}=\begin{bmatrix}x + \lambda h & (1/\lambda)\big((x+\lambda h)-x\big) \\0&x\end{bmatrix}.
\]
Hence, Lemma \ref{lem2.20} implies that
\be\label{3.50}
f(\begin{bmatrix}x + \lambda h & h \\0&x\end{bmatrix}) =
\begin{bmatrix}f(x + \lambda h) & (1/\lambda)\big(f(x+\lambda h)-f(x)\big) \\0&f(x)\end{bmatrix}.
\ee
As the left hand side of \eqref{3.50} is continuous at $\lambda = 0$, it follows that the 1-2 entry of the right hand side of \eqref{3.50} must converge. As $h$ is arbitrary after scaling, this implies that $f$ is holomorphic.
\ep

\section{$\tau$-holomorphic functions}

The theory of holomorphic functions on infinite dimensional spaces depends heavily on the topologies chosen --- 
see \eg \cite{din99}. For nc-functions too, the topology matters.
\bd
A topology $\tau$ on $\md$ is called admissible if it has a basis of bounded nc sets.
\ed
\bd
Let   $\tau$ be an admissible topology. A $\tau$-holomorphic function is an nc function on a $\tau$ open set that
is $\tau$ locally bounded.
\ed
By Theorem~\ref{thmhkm}, every $\tau$-holomorphic function on $\Omega$ is, at each level $n$, an analytic function from
$\O \cap \mnd$ to $\mn$.

\exam The {\em fine topology} is the admissible topology generated by all nc-sets. 
Since this is the largest admissible topology, for any admissible topology $\tau$,
any $\tau$-holomorphic function is automatically 
fine holomorphic. 
The class of nc functions considered in
\cite{hkm11b,pas14} is the 
fine holomorphic functions.

\vs
J. Pascoe proved the following inverse function theorem in \cite{pas14}.
The equivalence of (i) and (iii) is due to Helton, Klep and McCullough \cite{hkm11b}.

\bt
\label{thmb3}

Let  $\Omega \subseteq \md$ be an nc domain.
Let $\Phi$ be a fine holomorphic map on $\Omega$. Then the following are equivalent:

(i)  $\Phi$ is injective on $\Omega$.

(ii)
$D\Phi(a)$ is non-singular for every $a \in \Omega$.

(iii) The function $\Phi^{-1}$ exists and is  a  fine holomorphic  map.
\et

\exam
Let $\R^+ = \set{r \in \R}{r >0}$. For $n \in \N$, $a  \in \M_n^d$, and $r \in \R^+$, we let $D_n(a,r) \subseteq \M_n^d$ be the matrix polydisc defined by
\[
D_n(a,r) = \set{x \in \M_n^d}{\max_{1 \le i \le d}\|x_i -a_i \|  <r}.
\]
If $a \in \mnd$, $r \in \R^+$, we define $D(a,r) \subseteq \M^d$ by 
\[
D(a,r) = \bigcup_{k=1}^\infty D_{kn}(a^{(k)},r),
\]
where $a^{(k)}$ denotes the direct sum of $k$ copies of $a$. Finally, if $a \in \M^d$, $r \in \R^+$, we define $F(a,r) \subseteq \M^d$ by
\be\label{30}
F(a,r) = \bigcup_{m=1}^\infty\ \bigcup_{\ u \in \U_m} u^{-1} \big(D(a,r) \cap \M_m^d \big)\ u,
\ee
where $\U_m$ denotes the set of $m \times m$ unitary matrices.
It can be shown \cite{amif16} that the sets $F(a,r)$ are nc sets that form the basis for a topology.
We call this topology the {\em fat topology}.

There is an implicit function theorem for both the fine and fat topologies. The hypotheses require the
derivative to be full rank on a neighborhood of the point; the advantage of the fat topology is that if the derivative
is full rank at one point, it is automatically full rank on a neighborhood \cite[Thm. 5.5]{amif16}.

\bt
\label{thmc1}
Let $\Omega$ an nc domain.
Let $f = (f_1, \dots, f_k)^t$ be a vector of $k$ 
fine holomorphic function on $\Omega$, for some $1 \leq k \leq d-1$
(equivalently, a $\L(\C,\C^k)$ valued fine holomorphic function).
Suppose
\begin{eqnarray}
\nonumber
\lefteqn{
\forall\, n \in \N,\
\forall \, a \in \Omega \cap \mnd, 
}
 & \\
&\forall \, h 
\in \mn^k \setminus \{ 0 \}, \quad
Df(a) [(0,\dots,0,h^{d-k+1},\dots,h^d)] \neq 0.
\label{eqc1}
\end{eqnarray}
Let $W$ be the projection onto the first $d-k$ coordinates of $Z_f \cap \Omega$.
Then there is an $\L(\C,\C^k)$-valued fine holomorphic  function $g$ on $W$ such that
\[
Z_f \cap \Omega \=
\{ (y,g(y)) \, : \, y \in W \} .
\]

Moreover, if $f$ is fat holomorphic, then $g$ can also be taken to be fat holomorphic.
\et

\exam
The third example of an admissible topology is the 
 {\em free topology}.
A {\em basic free open set } in $\md$ is a set of the form
\[
\gdel \= 
\{ x \in \md : \| \d(x) \| < 1 \},
\]
where $\delta$ is a $J$-by-$J$ matrix with entries in $\pd$.
We define the free topology to be the topology on $\md$ which has as a basis all the sets $\gdel$,
as $J$ ranges over the positive integers, and the entries of $\delta$ range over all polynomials in $\pd$.
(Notice that $ G_{\delta_1} \cap G_{\delta_2} = G_{\delta_1 \oplus \delta_2}$, so these sets do form the basis of a topology). The free topology is a natural topology  when considering semi-algebraic sets.

\exam
Another admissible
 topology is the {\em  Zariski free topology}, which is generated by sets
of the form
\[
\{ x \in \gdel : f_i(x) \neq 0 \, \forall \, i \} , 
\]
where $\{ f_i \} $ 
are free holomorphic functions on $\gdel$.

There is no Goldilocks topology. The free topology has good polynomial approximation properties - 
not only is every free holomorphic function pointwise approximable by free polynomials, there is an Oka-Weil theorem
which says that on sets of the form $\gdel$, bounded free holomorphic functions are uniformly approximable on compact sets by free polynomials \cite{amfree}.
In contrast, neither the fine nor fat topology admit even pointwise polynomial approximation. Indeed,
one cannot have an admissible topology for which there is both an implicit function theorem and pointwise approximation by polynomials, as the following example shows (a formal statement is in \cite{amif16}).

\exam Let 
 $x_0 \in \M_2^2$ be 
\[
x_0 \= 
\left[
\begin{pmatrix} 
0&1\\
0&0
\end{pmatrix} ,
\
\begin{pmatrix} 
1&0\\
0&0
\end{pmatrix}  \right],
\]
and let $z_0 \in \M^2 $ be
\[
z_0 \= 
\begin{pmatrix} 
0&0\\
1&0
\end{pmatrix}.
\]

%

Define $p$ by
\be
\label{eqr13}
p(X,Y,Z) \= (Z)^2 + XZ + ZX +YZ - I.
\ee
If $x_0 = (X,Y)$ and $z_0 = Z$ are substituted in \eqref{eqr13},
we get $p(x_0,z_0) = 0$.
Let 
\be
\label{eqr14}
a \= 
\left[
\begin{pmatrix} 
0&1\\
0&0
\end{pmatrix} ,
\
\begin{pmatrix} 
1&0\\
0&0
\end{pmatrix}
\
\begin{pmatrix} 
0&0\\
1&0
\end{pmatrix}
\right],
\ee
We can calculate
\be
\label{eqr16}
\frac{\partial}{\partial Z} p(a) [h] \=
\begin{pmatrix} 
h_{11} + h_{12} + h_{21} & h_{11} + h_{12} + h_{22} \\
h_{11} + h_{22}   &    h_{12} + h_{21}
\end{pmatrix} .
\ee

It is immediate from \eqref{eqr16} that $ \frac{\partial}{\partial Z} p(a): \M_2 \to \M_2$ is onto, and so has a 
right inverse. 
By \cite[Thm 5.5]{amif16},
there is a fat domain $\Omega \ni a$
such that $\frac{\partial}{\partial Z} p(\lambda)$ is non-singular for all $\lambda \in \Omega$.
So by the Implicit Function Theorem \ref{thmc1},
if $W$ is the projection onto the first two coordinates
of $\Omega$, then $W$ is a fat domain containing $x_0$
and there is a fat holomorphic function $g$ defined on $W$ 
such that $g(x_0) =z_0$.
Since $z_0$ is not in the algebra generated by $x_0$, we cannot approximate $g$ pointwise by
free polynomials.

\section{Free holomorphic functions}
\label{secfree}

We shall let $H^\i(\gdel)$ denote the bounded free holomorphic functions on the polynomial polyhedron $\gdel$.
Just as in the commutative case \cite{at03, babo04}, there is a realization formula for these functions.
\bt \label{thmq1}
\cite[Thm. 8.1]{amfree}
Let $\delta$ be an $I$-by-$J$ matrix of free polynomials, and let $f$ be an nc-function on $\gdel$ that is bounded by $1$.
There exists an auxiliary Hilbert space $\L$ and an isometry
\[
  \begin{bmatrix}\alpha&B\\C&D\end{bmatrix} \ : \C \oplus \L^{I} \to \C \oplus \L^{J}
\]
so that for $x \in \gdel \cap B(\K)^d$, 
\be
\label{eqa44}
f(x) \= \alpha I_\K  +  (I_\K \otimes B) (\d(x) \ot I_\L)  \big[ I_\K \otimes I_{\L^J} - (I_\K \otimes D)  (\d(x) \ot I_\L) \big]^{-1}
(I_\K \otimes C).
\ee
\et
An immediate corollary is that if one truncates the Neumann series of $ \big[ I_\K \otimes I_{\L^J} - (I_\K \otimes D)  (\d(x) \ot I_\L) \big]^{-1}$ one gets a sequence of free polynomials that converges to $f$ uniformly on $G_{t \delta}$ 
for every $t > 1$.

\vs
The space $\K$ that appears in \eqref{eqa44} is finite dimensional, but the formula would make sense even it were not.
Let us fix an infinite dimensional separable Hilbert space $\h$.
Let us define $\gds$ by
\[
\gds \ := \ 
\{ x \in \lhd : \| \d (x) \| < 1 \} .
\]
If one extends formula \eqref{eqa44} to $\gds$, what sort of functions does one get?

\bd
\label{defa2}
Let $\Omega \subseteq \lhd$, and let $F: \Omega \to \lh$.
We say that $F$ is intertwining preserving (IP) if:

(i) Whenever $x,y \in \Omega$ and there exists
some bounded linear operator $T \in \lh$ such that $ T x = y T$, then $T F(x) = F(y) T$.

(ii) Whenever $( x_ n )$ is a bounded sequence  in $ \Omega$,  and there exists
some invertible bounded linear operator $ s :  \h \to  \oplus \h $ such that
\[
s^{-1} \begin{bmatrix} x_1 & 0 & \cdots\\
0 & x_2& \cdots\\
\cdots&\cdots&\ddots\end{bmatrix} s \in \Omega,
\] then
\[
F( s^{-1} \begin{bmatrix} x_1 & 0 & \cdots \\
0 & x_2 & \cdots \\
\cdots&\cdots&\ddots\end{bmatrix} s ) \=
s^{-1} \begin{bmatrix} F(x_1) & 0 & \cdots\\
0 & F(x_2) & \cdots \\
\cdots&\cdots&\ddots
  \end{bmatrix} s.
\]
\ed
\bd
\label{defa3}
Let $F: \bhd \to \bh$. We say $F$ is sequentially strong operator continuous (SSOC) if,
whenever $x_n \to x$ in the strong operator topology on $\bhd$, then
$F(x_n)$ tends to $F(x)$ in the strong operator topology on $\bh$.
\ed

Since multiplication is sequentially strong operator continuous, it follows that every free
polynomial is SSOC, and  this  property is also inherited by  limits on  sets that are closed w.r.t. direct sums.
The following theorem from \cite{amip15} characterizes what functions arise in \eqref{eqa44}:
\bt
\label{thma1}
 Assume that 
$\gds$ is connected and contains $0$. Let $F : \gds \to \lho$ be sequentially strong operator continuous. Then the following are equivalent:

(i) The function $F$ is intertwining preserving.

(ii) For each $t > 1$, the function $F$ is uniformly approximable by free polynomials on
$G_{t \d}^\sharp$.


(iii) There exists an auxiliary Hilbert space $\L$ and an isometry
\[
  \begin{bmatrix}\alpha&B\\C&D\end{bmatrix} \ : \C \oplus \L^{I} \to \C \oplus \L^{J}
\]
so that for $x \in \gds$, 
\be
\label{eqax4}
F(x) \= \alpha I_\h  +  (I_\h \otimes B) (\d(x) \ot I_\L)  \big[ I_\h \otimes I_{\L^J} - (I_\h \otimes D)  (\d(x) \ot I_\L) \big]^{-1}
(I_\h \otimes C).
\ee
\et

It could be the case that a free holomorphic function on $\gdel$ could have more than one extension to $\gds$. For example, if $\delta(x^1,x^2) = x^1 x^2 - x^2x^1 - 1 $, then $\gdel$ is empty, but $\gds$
is not.
The following theorem is from \cite{amip15}.
\bt
Assume that 
$\gds$ is connected and contains $0$. Then every bounded free holomorphic function on
$\gdel$ has a unique extension to an IP SSOC function on $\gds$.
\et

What about applying free holomorphic functions to $d$-tuples of operators on a Banach space, $\L(X)$?
This would allow the development of a non-commutative functional calculus. This is considered in \cite{amncf}.
The basic idea is to take a reasonable cross-norm $h$ defined on $X \otimes \K$ for every Hilbert space $\K$.
Let $(T) = (T_{ij})$ be an $I$-by-$J$ matrix of elements of $\L(X)$. Let $E_{ij} : \C^J \to \C^I$ be the matrix with $1$ in the $(i,j)$ entry and zero elsewhere.
Then we get a norm on $(T)$ by
\be
\label{eqnh}
\| (T) \|_h \ := \ \sup_{\K} \| \sum_{i,j} T_{ij} \otimes_h ( E_{ij} \otimes \id_\K) \|
\ee
where the $\ot$ without a subscript means the Hilbert space tensor product.
We define 
\[
\| (T) \|_\bullet \ := \ \inf_h \| (T) \|_h .
\]
Then provided $\| \delta(T) \|_\bullet < 1$, we can evaluate $f(T)$, when $f $ is in $H^\i(\gdel)$, by  a formula
like \eqref{eqa44} (or \eqref{eqax4}). Since tensor products need not be associative, we do all the Hilbert space tensor products first, and then tensor with the $\L(X)$ operator. Here is one theorem from \cite{amncf}.
\bt
\label{thmd01}
Suppose $\d(0) = 0$, and that
$T \in \L(X)^d$ has 
\[
\sup_{0 \leq r \leq 1} \| \d(r T) \|_\bullet < 1 .
\]
Then there exists a unique completely bounded homomorphism $\pi : H^\i(\gdel) \to \L(X)$
that extends the evaluation map on free polynomials $p \mapsto p(T)$.
\et

\bibliography{references}

\end{document}